\documentclass[11pt,oneside]{article}
\usepackage{amsthm}
\usepackage{amsfonts}
\usepackage{amssymb}
\usepackage{amsmath}
\usepackage{makeidx}  
\usepackage{hyperref}
\usepackage{graphicx}
\usepackage{graphics}
\usepackage{yfonts}
\usepackage{here}

\usepackage{centernot}

\newtheorem{theorem}{Theorem}[section]
\newtheorem{lemma}{Lemma}[section]

\newtheorem{corollary}{Corollary}[section]

\newtheorem{definition}{Definition}[section]

\theoremstyle{remark}

\usepackage{tikz}

\def\to{\rightarrow}

\def\Z{\mathbb Z}
\def\Q{\mathbb Q}

\def\C{\mathbb C}
\def\R{\mathbb R}

\def\to{\rightarrow}

\def\X{\mathcal{X}}

\begin{document}

\title{On CT and CSA Groups and Related Ideas}

\author{Benjamin Fine\\  
Department of Mathematics\\
Fairfield University\\Fairfield, Connecticut 06430\\
United States 
\and
Anthony Gaglione\\
Department of Mathematics\\
United States Naval Academy\\
Annapolis, Maryland 21402\\
United States
\and
Gerhard Rosenberger\\ 
Fachbereich Mathematik\\
University of Hamburg\\
Bundesstrasse 55\\
20146 Hamburg,
Germany
\and
Dennis Spellman\\
Department of Statistics\\
Temple University\\
Philadelphia, Pennsylvania 19122\\
 United States}

\date{ }

\maketitle

\begin{abstract} 
A group is $G$ {\bf commutative transitive} or CT if commuting is transitive on nontrivial elements. A group $G$ is CSA or {\bf conjugately separated abelian} if maximal abelian subgroups are malnormal. These concepts have played a prominent role in the studies of fully residually free groups, limit groups and dicriminating groups. They were especially important in the solution to the Tarski problems. CSA always implies CT however the class of CSA groups is a proper subclass of the class of CT groups. For limit groups and finitely generated elementary free groups they are equivalent. In this paper we examine the relationship between the two concepts. In particular we show that a finite CSA group must be abelian. If $G$ is CT then we prove that $G$ is not CSA if and only if $G$ contains a nonabelian subgroup $G_0$ which contains a nontrivial abelian subgroup $H$ that is normal in $G_0$. For $K$ a field the group $PSL(2,K)$ is never CSA but is CT if char$(K) = 2$ and for fields $K$ of characteristic $0$ where $-1$ is not a sum of two squares in $K$. For characteristic $p$, for an odd prime $p$, $PSL(2,K)$ is never CT. Infinite CT groups $G$ with a composition series and having no nontrivial normal abelian subgroup must be monolithic with monolith a simple nonabalian CT group. Further if a group $G$ is monolithic with monolith $N$ isomorphic to $PSL(2,K)$ for a field $K$ of characteristic $2$ and $G$ is CT then $G \cong N$.  
\end{abstract}


\bigskip

\section{Introduction} 

A group $G$ is {\bf commutative transitive}, which we will abbreviate by CT, if commutativity is
transitive on nonidentity elements.  Commutative transitivity is a simple idea that suprisingly has
had a wide-ranging impact on many areas of algebra in general and group theory in particular.  Of
special interest is the important role that commutative transitivity has played in the solution
of the celebrated Tarski conjectures. The paper [FR] contains a great deal of information about CT groups in general.

A group $G$ is CSA or {\bf conjugately separated abelian} if maximal abelian subgroups are malnormal (see section 2). CSA implies CT (see section 2) but the class of CSA groups is a proper subclass of the class of CT groups. These two concepts and their relationship have played a major role in the proof of the celebrated Tarski problems (see section 2). A result of Gaglione and Spellman [GS] and independently Remeslennikov [Re] showed that for nonabelian residually free groups, being CT is equivalent to having the same universal theory as a
nonabelian free group (see section 2). This result was one of the initial important steps in the
solution of the Tarski conjectures (see Section 2).

The term {\bf commutative transitive} was coined in [F] relative to free groups and Fuchsian
groups yet the concept appeared in the literature substantially earlier. In some papers a CT
group is referred to as {\bf centralizer abelian} or CA-group since being CT is easily shown to
be equivalent to having all centralizers of nontrivial elements abelian.   

Finite CT groups were studied originally by Weisner [W] in 1925.  He proved that finite CT groups
are either solvable or simple. However there was a mistake in his proof. Yu-Fen Wu in 1997 [Wu]
corrected the mistake and reproved Weisner's result.  She also proved that a finite solvable CT
group is the semidirect product of its Fitting subgroup $F$, which must be abelian. by a fixed point free group of
automorphisms of $F$. Earlier Suzuki [Su], in 1957, using character theory proved that every finite
nonabelian simple CT group is isomorphic to some $PSL(2,2^f), f \ge 2$. 

In this paper, we examine the relationship between the two concepts, CT and CSA. In the next section we review some important material on CT and CSA groups. As mentioned above,  
CSA implies CT however there do exist groups, both finite and infinite which are CT but not CSA. For limit groups, however, as well as elementary free groups and some related groups, the two concepts are equivalent. We provide a quick proof of this for limit groups in  section 3. 

We next consider finite CSA groups and prove, using the results of Wu, that a finite CSA must be abelian. Hence a finite CT group that is not simple and not CSA must have a nontrivial abelian normal subgroup.  For infinite groups we prove that a group $G$ that has a composition series and is CT but not CSA either contains a nontrivial normal abelian subgroup or is monolithic with monolith isomorphic to $PSL(2,K)$ for a field of characteristic $2$. Here we use the fact that CSA is given by a set of universal sentences and hence is true if and only if it is true in subgroups. The equivalence of CT and CSA carries over to the class of $B\X$-groups introduced by Ciobanu, Fine and Rosenberger [CFR]. 

\section{Basic Material on CT and CSA Groups}

A group G is {\bf commutative transitive} or CT if commutativity is transitive on nontrivial
elements.  That is
$$ [x,y]=1 \text { and } [y,z] =1 \implies [x,z] =1$$
provided $x,y,z$ are nontrivial.

It is straightforward that being commutative transitive is equivalent to the property
that the centralizer
of every nontrivial element is abelian.
For this reason CT groups are sometimes called CA groups or centralizer-abelian groups.

 $G$ is CT if and only if it satisfies the universal
sentence
$$
\forall x,y,z(((y\neq1)\wedge(xy=yx)\wedge(yz=zy))\rightarrow(xz=zx)).
$$ 
It is also clear is that if $Z(G) \ne \{1\}$ and $G$ is CT then $G$ is abelian.

It is clear (and follows directly from the fact that CT is captured by a universal sentence) that subgroups of CT groups are CT.

\begin{lemma} If $G$ is CT then any subgroup of $G$ is also CT.
\end{lemma}
Harrison [Ha] first published the following lemma that ties together the CT property with abelian
centralizers.

\begin{lemma} [Ha] Let $G$ be a group. The following three
statements are pairwise equivalent.

(i ) G is commutative transitive.

(ii ) The centralizer $C(x)$ of every nontrivial element in $x \in G$ is abelian.

(iii) Every pair of distinct maximal abelian subgroups in $G$ has trivial
intersection.
\end{lemma}

Finite CT groups were studied originally by Weisner [W] in 1925 
who proved that finite CT groups are either solvable or simple.  However there was a mistake in his
proof that was corrected by Yu-Fen Wu in 1997 [Wu].  She further proved that a finite solvable CT
group is the semidirect product of its Fitting subgroup $F$, which must be abelian, by a fixed
point free group of automorphisms of $F$.  Suzuki [Su] in 1957 using character theory proved that
every finite nonabelian simple CT group is isomorphic to some $PSL(2,2^f), f \ge 2$.

Wu further developed a complete structure theory for
locally finite CT groups analogous to that of finite CT groups. Her main result that is important for this paper is:

\begin{theorem} [Wu] An insolvable locally finite group is CT if and only if $G \cong
PSL(2,F)$ for some locally finite field $F$ of characteristic 2 with $|F| \ge 4$,
\end{theorem}

\begin{corollary} A finite CT group is either solvable or simple and isomorphic to $PSL(2,K)$ for some finite field of characteristic 2.
\end{corollary}

There are many examples of classes of CT groups.
  
$\hphantom{xx}$ (1) Free groups: It is well known that centralizers of elements in
nonabelian free groups are all cyclic (see [MKS] or [LS]).  This is usually expressed by saying
that two elements commute only if they are powers of a common element.

$\hphantom{xx}$ (2) Torsion-free hyperbolic groups: Again here centralizers are
cyclic (see [FGMS1]). 

$\hphantom{xx}$ (3) Free solvable groups: This was proved by Wu in the paper cited
above [Wu].  She also proved that there are solvable CT groups of any derived length.

$\hphantom{xx}$ (4) Free metabelian groups: Again see [Wu].

$\hphantom{xx}$ (5) $PSL(2,\R)$  where $\R$ is the real field: (see the next section)

Further from the observation that $PSL(2,\R)$ is CT and the obvious fact that all subgroups of
CT groups are also CT it follows that all {\bf Fuchsian groups} are
CT.  In particular any orientable surface group $S_g$ with $g \ge 2$ must be CT. This is important
relative to the tie between CT groups and residually free groups.

We note that $PSL(2,\C)$ and more generally $PSL(2,K)$ with $K$ an algebraically closed field of characteristic $0$ is never CT (see next section). 

$\hphantom{xx}$ (6) $PSL(2,K)$  where $K$ is a field with $char(K) = 2$: We will prove this in the next section using model theory. This result will be important in examining the relationship between CT and CSA. 

We note that if $K$ is any field of characteristic $p \ne 2$ then $PSL(2,K)$ is {\bf not} CT. This follows from Wu's result on finite CT groups and we give a proof in the next section.   

$\hphantom{xx}$ (7) Tarski groups:  These are simple groups where every proper subgroup
is cyclic hence centralizers are cyclic.
\smallskip

$\hphantom{xx}$ (8) One-relator groups with torsion:  This is a consequence of
results of Pride on two-generator one-relator groups coupled with earlier results of B.B.
Newman and Ree and Mendelsohn (see [FR] and the references there).
\bigskip  

Centralizers of elements and abelian subgroups in general are relatively well understood
relative to group amalgams, that is free products with amalgamation and HNN groups. Using these
ideas, commutative transitivity was studied relative to certain group amalgams by Levin and
Rosenberger [LR].
 
\begin{theorem} [LR] If $G_1$, $G_2$ are CT groups and $H$ is malnormal and proper
in both then the amalgamated product $G_1 \star_H G_2$ is also CT. \end{theorem}

They note that the malnormality condition cannot be relaxed. For example if the values $p,q > 1$ then the
amalgamated product $$\langle a \rangle \star_K \langle b,c;[b,c] = 1 \rangle$$ where $K = \langle a^p \rangle = \langle b^q \rangle$ is not CT.  

The situation for HNN extensions of CT groups is not as general but can be carried through for
extensions of centralizers of abelian malnormal subgroups.  

\begin{theorem} Let $B$ be a CT group and $K$ an abelian malnormal subgroup of $B$.  Then
the HNN extension 
$$B_1 = \langle  t,B ; rel(B), t^{-1}kt = k \text { for all } k \in K  \rangle$$
is a CT-group.
\end{theorem}

Myasnikov and Remeslennikov in their study of fully residually free groups introduced the concept
of a CSA group (conjugately separated abelian group).  Recall that if $G$ is a group and $H$ a subgroup of $G$ then $H$ is
{\bf malnormal} in $G$ or {\bf conjugately separated} in $G$ provided $g^{-1} Hg
\cap H = 1$  unless $g \in H.$  

Using this we define the concept of a CSA group.

\begin{definition}  A group $G$ is a {\bf CSA-group} or {\bf conjugately separated abelian
group} provided the maximal abelian subgroups are malnormal. 
\end{definition}

Each CSA group must be CT. The converse however is not true in general.
 
\begin{lemma} The class of CSA groups is a proper subclass of the class of CT groups.
\end{lemma}

\begin{proof} We first show that every CSA-group is commutative transitive. Let $G$ be a group in
which maximal abelian subgroups are malnormal and suppose that  $M_1$ and $M_2$ are maximal
abelian subgroups in  $G$ with $z \ne1$ lying in $M_1 \cap M_2$. Could we have $M_1 \ne M_2?$ 
Suppose  that $w \in M_1 \setminus M_2$.  Then $w^{-1}zw = z$ is a non-trivial element of
$w^{-1} M_2w \cap M_2$ so that $w \in M_2$. This is impossible and therefore $M_1 \subset  M_2$ .
By maximality we then get $M_1 = M_2$.  Hence, G is commutative transitive whenever all maximal
abelian subgroups are malnormal. 

We now show that there do exist CT groups that are not CSA. In any non-abelian CSA-group the only
abelian normal subgroup is the trivial subgroup 1.  To see this suppose that $N$ is any normal
abelian subgroup of the non-abelian CSA-group $G$. Then $N$ is contained in a
maximal abelian subgroup $M$. Let $g \notin M$. Then
$$N = g^{-1}Ng \cap  N \subset g^{-1} Mg  \cap M.$$ The fact $N \ne 1$ would imply that
$g \in M$ which is a contradiction. 

Now let $p$
and $q$ be distinct primes with $p$ a divisor of $q - 1.$  Let $G$ be the
non-abelian group of order $pq$ . Then it is not difficult to prove that the
centralizer of every non-trivial element of $G$ is cyclic of order either $p$ or
$q$ . Thus $G$ is commutative transitive.  However, the (necessarily unique)
Sylow q-subgroup of $G$ is normal in $G$. Hence from the argument above $G$ cannot be CSA.
\end{proof}

In the next section we give many more examples of CT but non CSA groups.
 
Although the class of CSA groups is a proper subclass of the CT groups, in the presence of full
residual freeness (in fact even in the presence of just residual freeness) they are equivalent (see the next section). Fully residually free groups play a prominent role in the solution of the Tarski problems (see [FGMRS]). Finitely generated fully residually free groups are also known as {\bf limit groups} since they arise (as initially observed by Sela [Se 1-5]) as limits of homomorphisms into free groups.  

\begin{definition} A group $G$ is {\bf residually free} if for each non-trivial $g \in G$
there is a free group $F_g$ and an epimorphism $h_g:G \to F_g$ such that $h_g(g) \ne 1$. 
Equivalently for each $g \in G$ there is a normal subgroup $N_g$ such that $G/N_g$ is free and $g
\notin N_g$. 

The group $G$ is {\bf fully residually free} provided to every finite set $S \subset
G\setminus\{1\} $ of non-trivial elements of $G$ there is a free group $F_S$ and an epimorphism
$h_S:G \to F_S$ such that $h_S(g) \ne 1$ for all $g \in S$.
\end{definition} 

There is a beautiful theorem due independently to Gaglione and Spellman [GS]
and Remeslennikov [Re] tying together full residual freeness, CT and and the property of being universally
free which we will explain shortly. 

In the 1960's G. Baumslag [GB] proved that a surface group is residually free, answering a
question of Magnus.  To do this he introduced what is now called {\bf extensions of centralizers}.
This concept became one of the main tools used by Kharlampovich, Myaasnikov and
Remeslennikov in their structure theory of fully residually free groups and by Kharlampovich and
Myasnikov in their solution to the Tarski problems.  Using some of G.Baumslag's techniques, B. Baumslag proved [BB]

\begin{theorem} If $G$ is residually free then the following are equivalent:

$\hphantom{xxxx}$ (1) $G$ is fully residually free

$\hphantom{xxxx}$ (2) $G$ is CT
\end{theorem}

Gaglione and Spellman and independently Remeslennikov developed a truly amazing result that in some sense is the beginning of the
solution of the Tarski problems. We first give some ideas from logic and model theory and then a brief introduction to the {\bf Tarski problems}.

We start with a first-order language appropriate for group theory. This language, which we denote by
$L_0$, is the first-order language with equality containing a binary operation symbol $\cdot$, a
unary operation symbol $^{-1}$ and a constant symbol $1$. A {\bf sentence} in this language is a
logical expression containing a string of variables $\overline{x} = (x_1,...,x_n)$, the logical
connectives $\lor,\land,\sim$ and the quantifiers $\forall,\exists$.  A {\bf universal sentence}
of $L_0$ is one of the form $\forall{\overline{x}} \{\phi(\overline{x})\}$ where $\overline{x}$ is a
tuple of distinct variables, $\phi(\overline{x})$ is a formula of $L_0$ containing no quantifiers and
containing at most the variables of $\overline{x}$.  Similarly an {\bf existential sentence} is one
of the form $\exists{\overline{x}}\{\phi(\overline{x})\}$ where $\overline{x}$ and
$\phi(\overline{x})$ are as above. 

If $G$ is a group then the {\bf universal theory} of $G$ consists of the set of all universal
sentences of $L_0$ true in $G$. We denote the universal theory of a group $G$ by $Th_\forall(G)$. 
Since any universal sentence is equivalent to the negation of an existential sentence it follows that
two groups have the same universal theory if and only if they have the same {\bf existential theory}.
The set of all sentences of $L_0$ true in $G$ is called the {\bf first-order
theory} or the {\bf elementary theory} of $G$.  We denote this by $Th(G)$. We note that being
{\bf first-order} or {\bf elementary} means that in the intended interpretation of any formula
or sentence all of the variables (free or bound) are assumed to take on as values only
individual group elements - never, for example, subsets of, nor functions on, the group in which
they are interpreted. The {\bf Tarski conjectures} or {\bf Tarski problems}, solved independently by Kharlampovich and Myasnikov
(see [KhM 1-5]) and Sela (see [Se 1-5]), say essentially that all countable nonabelian free
groups have the same elementary theory.  The following was well-known and much simpler.

\begin{theorem} All nonabelian free groups have the same universal theory.
\end{theorem}

A {\bf universally free group} $G$ is a group that has the same universal
theory as a nonabelian free group and, as we will see, all nonabelian finitely generated fully residually free groups are universally free.

One of the fundamental first steps in handling the Tarski problems was a theorem of Gaglione and Spellman [GS] and independently Remeslennikov [Re] who were able to extend the theorem of
B. Baumslag to show that fully residually free is equivalent to universally free and that these (in
the presence of residual freeness) are equivalent to both being CT and being CSA.  

\begin{theorem} [GS],[Re] If a nonabelian group $G$ is residually free then the following are equivalent:

$\hphantom{xxxx}$ (1) $G$ is fully residually free

$\hphantom{xxxx}$ (2) $G$ is CT

$\hphantom{xxxx}$ (3) $G$ is CSA

$\hphantom{xxxx}$ (4) $G$ is universally free.
\end{theorem}

Since this result  a complete structure theory and algorithmic theory of the fully residually free
groups has been developed. An important aspect of this development is that elements of fully
residually free groups can be expressed as {\bf infinite words} on a generating system.  These
infinite words can be manipulated and handled in an analogous manner to ordinary words in free
groups (see [FGMRS]).

The positive solution to the Tarski Problems (see [KhM1-5],[Se 1-5] and [FGMRS]) is given in the next three theorems:

\begin{theorem} (Tarski 1) Any two nonabelian free groups are elementarily
equivalent. That is any two nonabelian free groups satisfy exactly the same first-order
theory.  \end{theorem}

\begin{theorem} (Tarski  2) If the nonabelian free group $H$ is a free factor
in the free group
$G$ then the inclusion map $H \rightarrow G$ is an elementary embedding (see [FGMRS] for a precise definition of elementary embedding).
\end{theorem}

The question also arises concerning the {\bf decidability} of the theory of the nonabelian free groups,   The {\bf decidability } of the theory of nonabelian free groups
means the question of whether there exists a recursive algorithm which, given a sentence $\phi$ of
$L_0$, decides whether or not $\phi$ is true in every nonabelian free group. Kharlampovich and Myasnikov, in addition to proving the two above Tarski conjectures, also proved the following. 

\begin{theorem} (Tarski 3) The elementary theory of the nonabelian free
groups is decidable.
\end{theorem}

Commutative transitivity becomes essential in  building examples of fully residually free
groups via  the following
construction.

\begin{definition} Let $G$ be a CT group, let  $u\in G \setminus \{1\}$
and let $M =
Z_G(u)$ where $Z_G(u)$ is the centralizer of $u$ in $G$. Suppose $A$ is an
abelian group.  Then
the group
$$ H = \langle G,A;\text { rel } (G), \text { rel } A, [A,z] = 1,   \forall z \in M \rangle$$
is a {\bf centralizer extension} of $G$ by $A$.  If $A = \langle t \rangle$ is cyclic then $H
= G(u,t)$ is the
HNN extension
  $$G(u,t) = \langle G,t; \text { rel } (G), t^{-1} zt = z, \text { for all }z \in M \rangle$$
and is called the
{\bf free rank one extension of the centralizer} $M$ of $u$ in $G$ .
\end{definition}

\begin{lemma} (Baumslag, Myasnikov, Remeslennikov
[BMR3]) Let $G$ be a fully residually free group and $A$ an
abelian fully residually free group. Then a centralizer extension
of $G$ by $A$ is again fully residually free.
\end{lemma}

The proof of this result which is fundamental in all further
considerations of fully residually free groups depends on the fact
that the result can be reduced to free rank one extensions of
centralizers and then on the following "big powers" argument (originally developed by G.Baumslag in [GB]). It
is not hard to see that
 in a free group $F$ if
$$b_0t^{n_1}b_1...t^{n_k}b_k = 1$$ for infinitely many values of
$n_1$, infinitely many values of $n_2,...,$ infinitely many values
of $n_k$ then $t$ must commute with at least one of $b_0,...,b_k$.
Hence the family of homomorphisms $\phi_k : F(u,t) \rightarrow F$
from the rank one extension of the centralizer $C_F(u)$ into $F$,
defined for every positive $k$ by $\phi(t) = u^k$ and
$\phi_k{\mid_F} = id$, is a discriminating family, as required.

The class of finitely generated fully residually free groups was introduced in a different
direction by Sela in his proof of the Tarski problems.  In Sela's approach these groups appear as
limits of homomorphisms of a group $G$ into a free group.  In this guise they are called {\bf
limit groups}. Therefore a limit group is a finitely generated fully residually free group (see [FGMRS] for a proof of the equivalence of the two approaches.

As a by-product of the positive solution of the Tarski conjecture it was proved that the class of non-free groups that have exactly the same first order theory as the class of nonabelian free groups was nonempty. Such groups are called {\bf elementary free groups} (or {\bf elementarily free groups}) and both sets of authors provide complete characterizations of the finitely generated instances of them. In the Kharlampovich-Myasnikov approach these are the special NTQ-groups (see [KhM 1-5]). The primary examples of such groups are the orientable surface groups $S_g$ of genus $g \ge 2$ and the nonorientable surface groups $N_g$ of genus $g \ge 4$. That these groups are elementary free provides a powerful tool to prove some results in surface groups that are otherwise quite difficult. For example J.Howie [Ho] and independently O. Bogopolski [Bo], [BoS] proved that a theorem of Magnus about the normal closures of elements in free groups holds also in surface groups of appropriate genus. Their proofs were nontrivial.  However it was proved (see [FGRS] and [GLS]) that this result is first order and hence automatically true in any elementary free group. In [FGRS] a large collection of such results was given. Such results were called {\it something for nothing results}. Of course any such first order result true in a nonabelian free group must hold in any elementary free group.

\section{The Relationship Between CT and CSA}

As mentioned, CSA always implies CT but the class of CSA groups is a proper subclass of the class of CT groups. In this section we prove that $PSL(2,K)$ is never CSA. However if $K$ has characteristic $2$ then $PSL(2,K)$ is always CT while $PSL(2,\R)$ and $PSL(\Q_p)$ are also CT. The groups $PSL(2,K)$ for characteristic an odd prime $p$ are never CT. $PSL(2,\C)$ and more generally $PSL(2,K)$ where $K$ is an algebraically closed field of characteristic $0$ is never CT. Thus several of these types of groups provide an infinite number of example, both finite and infinite of CT non CSA groups. We also prove that a finite CSA group must be abelian. Wu [Wu] proves that there exist finite solvable CT groups for every solvability class. Hence the nonabelian ones provide more examples of CT non CSA groups. 

We first consider some cases where CT and CSA are equivalent.  

\begin{lemma} If $G$ is residually free then CT $\cong$ CSA.
\end{lemma}

\begin{proof} We assume Benjamin Baumslag's Theorem. CSA always imples CT so we assume that $G$ is CT and show that it must be CSA. From Baumslag's theorem since $G$ is residually free and CT it must be fully residually free so that we can assume that $G$ is a fully residually free group with more than one element. Let $u \in G \setminus \{1\}$ and let $M$ be its centralizer which we will denote by $C_G(u)$. Then
$M$ is maximal abelian in $G$. We claim that $M$ is malnormal in $G$. If $G$ is abelian, then
$M = G$ and the conclusion follows trivially. Suppose that $G$ is non-abelian.
Suppose that $w = g^{-1} zg \ne1$ lies in $g^{-1} Mg \cap M $.  If $g \notin M$
then $[g,u] \ne 1$. Thus, there is a free group $F$ and an epimorphism $h:G \to
F, x \to \overline{x}$, such that $ \overline{ w} \ne 1$ and $[ \overline{ g},
\overline{ u}] \ne 1$ . Let $C = C_F(\overline{u})$ . Then $ \overline{ w} \in  \overline{
g}^{-1}C \overline{ g}  \cap C$.  However the maximal abelian subgroups in a
free group are malnormal. This implies $ \overline{ g}  \in C$, contradicting
$[ \overline{ g}, \overline{ u}] \ne 1$. This contradiction shows that
$g^{-1}Mg \cap M \ne 1$ implies $g \in M$ and hence the maximal abelian
subgroups in $G$ are malnormal.
\end{proof}

Ciobanu, Fine and Rosenberger [CFR] generalized Benjamin Baumslag's theorem to what are called the class of $B\X$-groups. A class of groups $\X$ satisfies the property $B\X$ if a group $G$ is fully residually $\X$ if and only if $G$ is residually $\X$ and CT.

With this definition B. Baumslag's original theorem says that the class of free groups $\mathcal{F}$ satisfies $B\mathcal{F}$.

In [CFR] it was proved that a class of groups $\X$ satisfies $B\X$ under very mild conditions and hence the classes of groups for which this is true is quite extensive. In any class of groups satisfying $B\X$ the properties CT and CSA are equivalent.

\begin{theorem} (see [CFR]) Let $\X$ be a class of groups such that each nonabelian $H \in \X$ is CSA. Let $G$ be a nonabelian and residually $\X$ group.  Then the following are equivalent

$\hphantom{xx}$ (1) $G$ is fully residually $\X$

$\hphantom{xx}$ (2) $G$ is CSA

$\hphantom{xx}$ (3) $G$ is CT

Therefore the class $\X$ has the property $B\X$.

\end{theorem}

It follows that a class of groups $\X$ satisfies $B\X$ if each nonabelian $H \in \X$ is CSA. Examples of $B\X$ classes abound. In particular in [CFR] the following are listed.

\begin{theorem} Each of the following classes satisfies $B\X$:

$\hphantom{xx}$ (1) The class of nonabelian free groups.

$\hphantom{xx}$ (2) The class of limit groups.

$\hphantom{xx}$ (3) The class of noncyclic torsion-free hyperbolic groups (see [GKM]).

$\hphantom{xx}$ (4) The class of noncyclic one-relator groups with only odd torsion (see [GKM]).

$\hphantom{xx}$ (5) The class of cocompact Fuchsian groups with only odd torsion.

$\hphantom{xx}$ (6) The class of noncyclic groups acting freely on $\Lambda$-trees where $\Lambda$ is an ordered abelian group (see [CFR])

$\hphantom{xx}$ (7) The class of noncylic free products of cyclics with only odd torsion (see [GKM])

$\hphantom{xx}$ (8) The class of noncyclic torsion-free RG-groups (see[FMgrRR] and [CFR]).

$\hphantom{xx}$ (9) The class of conjugacy pinched one-relator groups of the following form

$$G = \langle F,t; tut^{-1} = v\rangle $$

where $F$ is a free group of rank $n \ge 1$ and $u,v$ are nontrivial elements of $F$ that are not proper powers in $F$ and for which $\langle u\rangle  \cap  x\langle v\rangle x^{-1} = \{1\}$ for all $x \in F$.

\end{theorem}

For the rest of this section we will concentrate on the situations where CT and CSA are not equivalent. That is we will examine CT non CSA groups. We need some prelimiaries. We saw that CT is given by a universal sentence and hence is captured by subgroups. The same is true for CSA

\begin{lemma} If $G$ is a CSA group and $H \subset G$ then $H$ is CSA.
\end{lemma}

\begin{proof} We first give a direct proof (see [GKM]).  Let $G$ be a CSA group and let $H$ be a subgroup of $G$. Let $A_H$ be a maximal abelian subgroup of $H$. We must show that $A_H$ is malnormal in $H$. Let $x \in H$ with $xA_Hx^{-1} \cap A_H \ne \{1\}$. $A_H$ is contained in a maximal abelian subgroup $A_G$ of $G$. Since $G$ is CSA it follows that $A_G$ is malnormal in $G$ and so $x \in A_G$.  Then $x \in (A_G \cap H) \subset A_H$ and hence $A_H$ is malnormal in $H$.

The result also follows from the fact that CSA can also be described in terms of universal sentences. In particular the CSA property is described by the following pair of universal sentences.

$$\text(CT: ) \forall x,y,z (((y \ne 1) \wedge ([x,y] =1) \wedge ([y,z] = 1)) \rightarrow ([x,z] = 1))$$
$$\text(MAL: ) \forall x,y,z ((x \ne 1) \wedge (y \ne 1) \wedge  ([x,y] =1) \wedge ([x,z^{-1}yz]=1) \rightarrow ([y,z] = 1))$$ 

\end{proof}

Recall that the infinite dihedral group is the free product $D = \Z_2 \star \Z_2$. The group $D$ then has the presentation $D = \langle x,y; x^2 = y^2 = 1\rangle $.  Then $xxyx^{-1} = yxyy^{-1} = yx = (xy)^{-1}$ and hence $D$ is not CSA. 

\begin{lemma} If the group $G$ contains a subgroup isomorphic to the infinite dihedral group then $G$ is not CSA.
\end{lemma}

\begin{corollary} The modular group $M = PSL(2,\Z)$ is not CSA.
\end{corollary}

\begin{proof} The modular group $M$ is isomorphic to the free product $\Z_2 \star \Z_3$ of a cyclic group of order $2$ and a cyclic group of order $3$. Such a free product contains as a subgroup the free product $\Z_2 \star \Z_2$, a subgroup isomorphic to the infinite dihedral group. Therefore by Lemma 3.4 $M$ cannot be CSA. 
\end{proof}

\begin{corollary} The group $PSL(2,\Q)$ where $\Q$ is the field of rational numbers is not CSA.
\end{corollary}

\begin{proof} $PSL(2,\Q)$ contains $M$ as a subgroup and hence cannot be CSA.
\end{proof}

Because of Wu Fen's work on finite CT groups (see Theorem 2.1 ) the groups $PSL(2,K)$, where $K$ is a field, figure prominently in the analysis of CT and CSA groups. In [Wu] she also proves that there are finite solvable CT groups for every solvability class.

We prove the following two results. The first is that $PSL(2,K)$ for any field is never CSA.

\begin{theorem} Suppose that $K$ is a field. Then the group $PSL(2,K)$ is not CSA.
\end{theorem}

\begin{proof} We consider the characteristic of $K$ and handle each characteristic separately. 
If $K$ is a field of characteristic $p \ne 2$ then the group $PSL(2,K)$ is not CT and hence it cannot be CSA.

Now let $K$ be a field of characteristic $0$. Then $K$ contains a subfield isomorphic to $\Q$. Hence $PSL(2,K)$ contains a subgroup isomorphic to $PSL(2,\Q)$. From Lemma 3.3 $PSL(2,\Q)$ is not CSA and therefore $PSL(2,K)$ cannot be CSA.

Finally let $K$ be a field of characteristic $2$. Let $F = \Z_2$ be the two-element field. Then $K$ contains a subfield isomorphic to $F$ and hence $PSL(2,F) = PSL(2,\Z_2)$ is a subgroup of $PSL(2,K)$. However $PSL(2,\Z_2)$ is nonabelian of order $6$ and hence is isomorphic to $S_3$ the symmetric group on $3$ symbols. This group has an abelian normal subgroup or order $3$ and hence is not CSA. It follows that $PSL(2,K)$ cannot be CSA.  
\end{proof}

The next theorem handles the CT property for $PSL(2,K)$. It is more complex than for CSA.

\begin{theorem} Suppose that $K$ is a field. 

$\hphantom{xx}$ (1) If char$(K) = 2$ then the group $PSL(2,K)$ is CT. 

$\hphantom{xx}$ (2) If char$(K) = p$ where $p$ is an odd prime the group $PSL(2,K)$ is not CT.

$\hphantom{xx}$ (3) If char$(K) = 0$ then the group $PSL(2,K)$ is CT if $-1$ is not a sum of two squares in $K$ and not CT if $-1$ is a sum of two squares in $K$. In particular if $K = \R$, the real numbers or $K = \Q_p$ the p-adic numbers or any subfield of these then $PSL(2,K)$ is CT. On the other hand $PSL(2,\C)$ is not CT and more generally $PSL(2,K)$ is not CT for any algebraically closed field of characteristic $0$.
\end{theorem}

It follows from these two theorems that in the class of groups $PSL(2,K)$ there are infinitely many examples, both of finite order and infinite order of CT non CSA groups.

\begin{proof} (Theorem 3.4) We do each characteristic separately with characteristic $p$ the simplest. 

\begin{lemma} If $K$ is a field of characteristic $p \ne 2$ then the group $PSL(2,K)$ is not CT.
\end{lemma}

\begin{proof} From Wu's result a finite CT group must either be solvable or isomorphic to $PSL(2,K)$ where $K$ is a field of characteristic 2. Hence if $p \ne 2$ we must have that $PSL(2,\Z_p)$ is not CT for the finite field $K = \Z_p$ . If $K$ is a field of characteristic $p \ne 2$ then $PSL(2,K)$ will contain $PSL(2,\Z_p)$ as a subgroup. Since the CT property is captured by subgroups it follows that $PSL(2,K)$ cannot be CT.
\end{proof}

\begin{lemma} If char$(K) = 0$ then the group $PSL(2,K)$ is CT if $-1$ is not a sum of two squares in $K$ and not CT if $-1$ is a sum of two squares in $K$. In particular if $K = \R$, the real numbers or $K = \Q_p$ the p-adic numbers or any subfield of these then $PSL(2,K)$ is CT. Further $PSL(2,K)$ is CT for any real field. On the other hand $PSL(2,\C)$ is not CT and more generally $PSL(2,K)$ is not CT for any algebraically closed field of characteristic $0$.
\end{lemma}

\begin{proof} Suppose that $K$ is a field with  char$(K) = 0$ and suppose that there do not exist elements $x,y \in K$ such that $x^2 + y^2 = -1$. Let $A,B,C$ be nontrivial elements of $PSL(2,K)$ with $AB = BA$ and $BC = CB$. Since $K$ can be embedded in an algebraic closure $k$ we may assume that each of $A,B,C$ has one or two eigenvalues within $k$. 

Case 1: $B$ has one eigenvalue in $k$. Then this eigenvalue is already in $K$. After a suitable conjugation in $PSL(2,K)$ we may assume that 
$$ B = \pm \left (\begin{matrix} 1&1\\0&1 \end{matrix} \right ).$$
Let
$$ A = \pm \left (\begin{matrix} a&b\\c&d\end{matrix} \right ) \text{ in } PSL(2,K)$$

From $AB = BA$ we get that either $c = 0$ or $a = d = -b, -c = 2a$. We must have $c = 0$ for otherwise this implies that
$$ 1 = ad - bc = a^2 - 2a^2 = -a^2 - 0^2$$
and hence $-1$ is a sum of two squares in $K$ contrary to assumption. Further $AB = BA$ then leads to $a = d = \pm 1$. Hence $A$ has the form
$$ A = \pm \left (\begin{matrix} 1&\alpha\\0&1\end{matrix} \right ) \text{ in } PSL(2,K).$$
An analogous statement holds for $C$ since $BC = CB$. Therefore $C$ has this form also and hence $AC = CA$ in Case 1.

Case 2: $B$ has two eigenvalues in $k$. After a suitable conjugation in $PSL(2,k)$ we may assume that 
$$ B = \pm \left (\begin{matrix} \alpha&0\\0&\alpha^{-1} \end{matrix} \right ) \in PSL(2,k).$$
Since $B$ is nontrivial we have $\alpha \ne \pm 1$.

Let
$$ A = \pm \left (\begin{matrix} a&b\\c&d\end{matrix} \right ) \text{ in } PSL(2,k).$$
Then from $AB = BA$ we get that either $b = c = 0$ or $c \ne 0, c\alpha = -c\alpha^{-1}$ or $b \ne 0, b\alpha = -\alpha^{-1}b$.

If $c \ne 0$ or $b \ne 0$ then $\alpha =-\alpha^{-1}$ and hence the trace, $tr(B) = 0$.

Analogously for $C$. Let
$$ C = \pm \left (\begin{matrix} e&f\\g&h\end{matrix} \right ) \text{ in } PSL(2,k).$$
Then from $BC = CB$ we get that either $f = g = 0$ or $tr(B) = 0$.

If $b = c = f = g = 0$ then $AC = CA$ and hence here in case 2 we may assume that $tr(B) = 0$. 

We now consider $A,B,C \in PSL(2,K)$ with $tr(B) = 0$. Let
$$B = \pm \left (\begin{matrix} \alpha&\beta\\\gamma&-\alpha \end{matrix} \right ).$$
We have that $\gamma \ne 0$ for if $\gamma = 0$ then $-\alpha^2 = -\alpha^2 + 0^2 = 1$ contrary to assumption that $-1$ is not a sum of squares. 

Hence by conjugation we may assume that $B$ has the form
$$ B = \pm \left (\begin{matrix} 0&1\\-1&0 \end{matrix} \right ).$$

Let
$$ A = \pm \left (\begin{matrix} a&b\\c&d\end{matrix} \right ).$$
Then from $AB = BA$ we get that either 
$$ A = \pm \left (\begin{matrix} x&-y\\y&x \end{matrix} \right ) \text{ or } A = \pm \left (\begin{matrix} -x&y\\y&x \end{matrix} \right ).$$
If
$$ A = \pm \left (\begin{matrix} -x&y\\y&x \end{matrix} \right )$$
then $$-x^2 - y^2 = 1$$
contradicting the assumption on $-1$ not being a sum of squares.
Therefore 
$$A = \pm \left (\begin{matrix} x&-y\\y&x \end{matrix} \right )$$

Analogously let
$$ C = \pm \left (\begin{matrix} e&f\\g&h\end{matrix} \right ).$$
Then from $BC = CB$ we get that 
$$C = \pm \left (\begin{matrix} u&-v\\v&u \end{matrix} \right )$$
But in this case $AC = CA$.

Therefore altogether $PSL(2,K)$ is CT if $-1$ is not a sum of two squares in $K$. 

Now suppose that $-1 = x^2 + y^2$ in $K$. Since $K$ has characteristic $0$ we have $\Q \subset K$. Let $\alpha,\beta$ be nonzero elements of $\Q$ such that $\alpha^2 + \beta^2 = 1$. For example let $\alpha = \frac{3}{5},\beta = \frac{4}{5}$. Now let
$$ A = \pm \left (\begin{matrix} -x&y\\y&x \end{matrix} \right )$$
$$ B = \pm \left (\begin{matrix} 0&1\\-1&0\end{matrix} \right )$$
and 
$$ C = \pm \left (\begin{matrix} \frac{3}{5}&\frac{4}{5}\\-\frac{4}{5}&\frac{3}{5} \end{matrix} \right ).$$
Then $A,B,C$ are three nontrivial elements of $PSL(2,K)$ with $AB = BA$, $BC = CB$ but $AC \ne CA$ so the group is not CT.

Notice that if $-1$ is itself a square in $K$ the group $PSL(2,K)$ then cannot be CT. In particular this is true for the complex numbers $\C$ and more generally for any algebraically closed field $K$. We give an example in $PSL(2,\C)$ to clarify this.
  
In $PSL(2,\C)$ we have the projective matrices
$$ T = \pm \left ( \begin{matrix} 0&1\\-1&0 \end{matrix} \right ), U = \pm \left ( \begin{matrix} i&0\\0&-i \end{matrix} \right ),
V = \pm \left ( \begin{matrix} \alpha&0\\0&\frac{1}{\alpha} \end{matrix} \right ), \alpha \ne \pm 1.$$
As linear fractional transformations these are
$$T:z' = -\frac{1}{z}, U:z' = -z, V: z' = \alpha^2 z.$$
By a direct computation $U$ commutes with $T$ and $V$ but $T$ and $V$ do not commute. Therefore $PSL(2,\C)$ is not CT.

Exactly the analogous example works in any field $K$ of characteristic $0$ where $-1$ is a square. Therefore the example holds in $PSL(2,K)$ for any algebraically closed field of characteristic $0$.

Notice further that the lemma applies to $PSL(2,\R)$ for the real numbers $\R$ and for any subfield of $\R$, in particular any algebraic number field, and for any subgroup of $PSL(2,\R)$. Hence any Fuchsian group is CT.

In general, fields where $-1$ is not a sum of squares are called {\bf real fields} and have been extensively studied (see [La]). If $-1$ is not a sum of squares then it is not a sum of two squares and hence $PSL(2,K)$ is CT for any real field.   
\end{proof}

\begin{lemma} Let $K$ be a field of characteristic $2$. Then $PSL(2,K)$ is CT. 
\end{lemma}

\begin{proof} Let $K$ be a field of characteristic $2$ and let $F = \Z_2$ be the two element field. Clearly $F$ is a subfield of $K$. For a field of characteristic $2$ we have $PSL(2,K) = SL(2,K)$ so we show that $SL(2,K)$ is CT.  Now $PSL(2,F) \cong S_3$, the symmetric group on three symbols. This group is CT so we now may assume that $|K| \ge 4$.

Let $\overline{K}$ be an algebraic closure of $K$ and let $k$ be the algebraic closure of $F$ in $\overline{K}$. Then we have the tower of fields. 

\begin{center}
\begin{tikzpicture}
\node (a) at (0:1) {$K$};
\node (b) at (90:1) {$\overline{K}$} edge [-] (a);
\node (c) at (180:1) {$k$} edge [-] (b);
\node (d) at (270:1) {$F$} edge [-] (a)
edge [-] (c);
\end{tikzpicture}
\end{center}
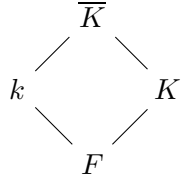
\centerline {Tower of Fields 1}

\smallskip

From [Hod, p. 47] we have that one form of Hilbert's Nullstellensatz says that if $A$ is an algebraically closed field and $E$ is a finite systems of equations and inequations with coefficients from $A$, such that some field extending $A$ contains a solution of $E$, then $A$ already contains a solution of $E$ (see also Jacobson [Ja], p. 425). It follows from this that an existentially closed field (see[Ja] for a definition) is the same thing as an algebraically closed field. We now use a bit of model theory. We refer the reader to [BeS] or [FGMRS] for a discussion of ultrapowers.

Since $k$ and $\overline{K}$ are algebraically closed they are existentially closed and hence, since $k \subset \overline{K}$ we must have the universal equivalence $k \equiv_\forall \overline{K}.$

By Lemma 3.8 in [BeS] (p. 187) the field $\overline{K}$ embeds in an ultrapower $^\star k = k^I/D$ of $k$. We now have the diagram;

\begin{center}
\begin{tikzpicture}
\node  (f) at (90:2) {$^\star k$} ;
\node (a) at (0:1) {$K$};
\node (b) at (90:1) {$\overline{K}$} edge [-] (a) edge [-] (f);
\node (c) at (180:1) {$k$} edge [-] (b);
\node (d) at (270:1) {$F$} edge [-] (a)
edge [-] (c);

\end{tikzpicture}
\end{center}
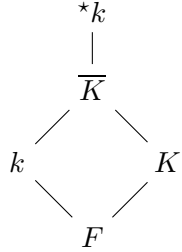
\centerline {Tower of Fields 2}

\smallskip

For a field $E$ making the obvious identification of $^\star SL(2,E) = SL(2,E)^I/D$ with $SL(2,^\star E) = SL(2,E^I/D)$ we have the diagram of group inclusions.

\smallskip

\begin{center}
\begin{tikzpicture}
\node (a) at (0:1) {$SL(2,K)$};
\node (b) at (90:1.5) {$^\star SL(2,k)$} edge [-] (a);
\node (c) at (180:1) {$SL(2,k)$} edge [-] (b);
\end{tikzpicture}
\end{center}
\centerline {Group Inclusions}

\smallskip

From the result of Wu we have that $SL(2,k)$ is CT since $k$ is locally finite. Further $^\star SL(2,k) \equiv_\forall SL(2,k)$ and hence $^\star SL(2,k)$ is CT since CT is expressed as a universal sentence in the language $L_0$ of group theory. But the CT property is inherited by subgroups and therefore it follows that $SL(2,K)$ is CT.
\end{proof}

These three lemmas complete the proof of Theorem 3.1. 

\end{proof}

We now prove:

\begin{theorem} Let $G$ be a finite CSA group. Then $G$ is abelian.
\end{theorem}

\begin{proof} Let $G$ be a finite CSA group. Since CSA implies CT we then have $G$ is a finite CT group. From Wu's theorem $G$ is then either solvable or isomorphic to $PSL(2,K)$ for a finite field of characteristic $2$. If $G \cong PSL(2,K)$ then from Theorem 3.2 $G$ cannot be CSA. Hence $G$ must be solvable. If the solvability class is $d > 1$ then the element of the derived series $G^{d-1}$ is a nontrivial abelian normal subgroup and hence $G$ cannot be CSA in this case. It follows that the solvability class must be $d = 1$ and therefore $G$ is abelian.
\end{proof}

\section{Infinite CT non CSA Groups}

 For finite groups a CT group is either solvable or simple. The situation for infinite CT groups is more From Theorem 3.3 and Wu's analysis we have that finite CSA groups are abelian. Wu proves that there exist finite CT groups of every solvability class and hence the nonabelian ones (those of solvability class $d > 1$ provide examples of finite CT non CSA groups. The situation for infinite CT groups is more complicated. First notice that the infinite Tarksi simple groups mentioned in previous section are both CT and CSA while the free product of two finite CT non CSA groups such as $S_3 \star S_3$ is an infinite CT but non CSA group. 

\begin{lemma} Let $G,H$ be any two CT non CSA groups. Then the free product $G \star H$ is also CT non CSA.
\end{lemma}

\begin{proof} The free product of CT groups is again CT so $G \star H$ is CT. However $G$ can be considered as a subgroup of the free product so $G \star H$ cannot be CSA.
\end{proof}

We now give several results characterizing infinite CT non CSA groups. Notice that if a group $G$ contains a nontrivial abelian normal subgroup then it cannot be CSA. This is almost enough to characterize when a CT group is not CSA. 

\begin{theorem} A CT group $G$ is not CSA if and only if $G$ contains a nonabelian subgroup $G_0$ which itself contains a nontrivial abelian subgroup $H$ which is normal in $G_0$
\end{theorem}

\begin{proof} Suppose that $G$ contains a nonabelian subgroup $G_0$ which itself contains a nontrivial abelian subgroup $H$ which is normal in $G_0$. Then $G_0$ cannot be CSA and hence $G$ cannot be CSA.

Conversely suppose that $G$ is CT but not CSA. Recall that in the presence of the group axioms the CSA property is captured by the pair of universal sentences

$$\text(CT: ) \forall x,y,z (((y \ne 1) \wedge ([x,y] =1) \wedge ([y,z] = 1)) \rightarrow ([x,z] = 1))$$
$$\text(MAL: ) \forall x,y,z ((x \ne 1) \wedge (y \ne 1) \wedge  ([x,y] =1) \wedge ([x,z^{-1}yz]=1) \rightarrow ([y,z] = 1))$$

It follows that being CT but not CSA is captured (in the presence of the group axioms) by

$$\text(CT: ) \forall x,y,z (((y \ne 1) \wedge ([x,y] =1) \wedge ([y,z] = 1) \rightarrow ([x,z] = 1))$$
$$\text(NOTMAL: ) \exists x,y,z ((x \ne 1) \wedge (y \ne 1) \wedge  ([x,y] =1) \wedge ([x,z^{-1}yz]=1) \wedge ([y,z] \ne 1))$$

Now suppose that $G$ is CT but not CSA. Let $g,h,k \in G$ such that if $g = x,h = y, z = k$ then these three elements verify NOTMAL in $G$.
Consider the subgroup $G_0 = \langle g,h,k \rangle$ of $G$. This is nonabelian since $[h,k] \ne 1$.

Now consider the subgroup $A = \langle h \rangle^{G_0}$, the normal closure of $\langle h \rangle$ in $G_0$. Since $h \ne 1$ the subgroup $A$ is nontrivial. We claim that $A$ is abelian which will complete the proof.

Now $g$ and $h$ commute with $h$ and further $g \ne 1$ commutes with $k^{-1}hk$ and commutes with $h$ so by CT, $k^{-1}hk$ commutes with $h$. From the fact that $k^{-1}hk$ commutes with $h$ we get that $k(k^{-1}hk)k^{-1}$ commutes with $khk^{-1}$ and so $khk^{-1}$ commutes with $h$. Hence if $u \in \{g,g^{-1},h,h^{-1},k,k^{-1}\}$ then $uh^\epsilon u^{-1}$ commutes with $h$ where $\epsilon = \pm 1$. It follows that these conjugates commute with each other.

If $A$ were not abelian there would be a word $w(x,y,z)$ of shortest length such that $$w(g,h,k)^{-1} h w(g,h,k)$$ did not commute with $h$. Choose such a $w$ with $w = v(x,y,z)u$ and $u \in \{g,g^{-1},h,h^{-1},k,k^{-1}\}$. Then by minimality $v(g,h,k)^{-1}h v(g,h,k)$ commutes with $h$. Let $\overline{u}$ be the value of $u$ in $G_0$ so that $\overline{u} \in \{g,g^{-1},h,h^{-1},k,k^{-1}\}$. From
$v(g,h,k)^{-1}h v(g,h,k)$ commuting with $h$ we get that $\overline{u}^{-1} v(g,h,k)^{-1}h v(g,h,k)\overline{u}$ commutes with $\overline{u}^{-1} h \overline{u}$. But $\overline{u}^{-1} h \overline{u}$ commutes with $h$ so by CT
$$ w(g,h,l)^{-1} h w(g,h,k) = \overline{u}^{-1} v(g,h,k)^{-1} h v(g,h,k) \overline{u}$$ 
commutes with $h$ contradicting our choice of $w(x,y,z)$. Therefore this contradiction shows that $A$ must be abelian.  
\end{proof} 

To proceed further we need the following concept. A group $G$ is {\bf monolithic} if $G$ contains a unique nontrivial minimal normal subgroup $N$ (see [N]). This subgroup is then called the {\bf monolith}. Our first result is the following:

\begin{theorem} Let $G$ be a nontrivial CT group which contains no nontrivial abelian normal subgroup. If $G$ has a composition series then $G$ is monolithic whose monolith $N$ is a simple nonabelian CT group.
\end{theorem}

\begin{proof} Notice that if the monolith $N \cong PSL(2,K)$ for $K$ a field of characteristic $2$, which is the situation for finite CT groups with no abelian normal subgroups then $G$ would not be CSA. However as pointed out above there do exists simple nonabelian CT groups that are CSA. 

If $H$ is a group then a descending chain of subgroups
$$ H = H_0 \supset H_1 \supset \cdots H_n$$
is a {\bf chief series} (see [Hal] P. 124) from $H$ to $H_n$ provided $H_i$ is normal in $H$ for all $i=0,1,...,n$ and for all $i=0,...,n$, $H_i$ is maximal normal in $H_{i-1}$. 

Let $G$ be a CT group with a composition series and no abelian normal subgroup. Since $G$ is assumed to have a composition series it follows from Theorem 8.6.1 ([Ha], p.131) that $G$ has a chief series
$$ G = G_0 \supset G_1 \supset \cdots \supset G_{n-1} \supset G_n = \{1\}.$$
Let $M = G_{n-1}$. Then $G_n = \{1\}$ is maximal normal in $M$ and hence there is no subgroup $N$ normal in $G$ such that $M \supset N \supset \{1\}$. It follows that $M$ is a minimal normal subgroup in $G$. 

We claim that $M$ is unique. Suppose that $M_1$ and $M_2$ are minimal normal subgroups of $G$ with $M_1 \ne M_2$. By assumption neither is abelian. By minimality $M_1 \cap M_2 = \{1\}$. It follows that the subgroup $H = \langle M_1,M_2 \rangle$ generated by $M_1,M_2$ is their direct product. That is $H = M_1 \times M_2$. However a direct product of nonabelian groups is not CT a contradiction since $G$ is CT and CT is inherited by subgroups. Therefore $M$ is a unique minimal normal subgroup and hence $G$ is monolithic with monolith $M$.

Again from Theorem 8.6.1 in [Ha] $M$ is a direct power $A^m$ of a simple group $A$. Since $G$ contains no normal abelian subgroup it follows that $A^m$ is not abelian and hence $A$ is not abelian. If $m > 1$ the monolith $A^m$ is not CT, again a contradiction and therefore $m = 1$ and the monolith is a nonabelian simple CT group.
\end{proof}

We now consider monolithic groups with monolith isomorphic to $PSL(2,K)$ for a field $K$ of characteristic $2$.

\begin{theorem} Let $G$ be a monolithic group with monolith isomorphic to $PSL(2,K)$ where $K$ is a field of characteristic $2$ with $|K| \ge 4$. Then if $G$ is CT we must have $G \cong PSL(2,K)$ and hence $G$ is non CSA.
\end{theorem}

We need two preliminary results before we prove this theorem.

\begin{lemma} Let $G$ be a monolithic group with monolith $M$ isomorphic to $PSL(2,K)$ where $K$ is a field of characteristic $2$. If $G$  is CT, $G$ then embeds into Aut$(M)$.
\end{lemma}

\begin{proof} Since $M$ is normal in $G$ we get a map $\phi$ from $G$ to Aut$(M)$ by mapping $g \in G$ to conjugation on $M$ by $g$. Now $M = SL(2,K)$ is nonabelian. Choose $a,b \in M$ such that $ab \ne ba$ and suppose that $z \in ker(\phi)$. Then $zaz^{-1} = a$ and $zbz^{-1}z = b$. Now $z$ commutes with both $a$ and $b$. If $z \ne 1$ and $ab \ne ba$ this contradicts the assumption that $G$ is CT. Hence $ker(\phi) = \{1\}$ and hence $\phi$ is an embedding.
\end{proof}

Recall that an algebraic structure is {\bf rigid} if it admits only the identity automorphism.

\begin{theorem} Let $K$ be a field of characteristic $2$ with $|K| \ge 4$. Then $K$ is not rigid.
\end{theorem}

\begin{proof} Let $K$ be a field of characteristic $2$ with $|K| \ge 4$. Then the map $\sigma:K \rightarrow K$ given by $\sigma(x) = x^2$ for all $x \in K$ is an injective homomorphism. If it is surjective we are done since the only roots of $x^2 - x$ over $K$ are $0$ or $1$ and we thus get a nontrivial automorphism.

Assume now that $K$ contains an element $\theta$ which is not a square in $K$ and assume that $K$ is rigid.  Now consider the simple group $M = SL(2,K)$. Recall that $SL(2,K) = PSL(2,K)$. By Theorem 15.3.2 in [Sc] we have that Aut$(M)$ is complete. Since we assumed that $K$ is rigid Aut$(M)$ consists solely of inner automorphisms and hence $M$ is isomorphic to Aut$(M)$ and hence $M$ is complete. By Theorem 15.3.1 in [Sc] if $H$ is any overgroup of $M$ in which $M$ is normal then $H$ has the internal direct product representation $M \times C_H(M)$ where $C_H(M)$ is the centralier of $H$ in $M$.

Thus since $M = SL(2,K)$ is normal in $H = GL(2,K)$, $H$ has the internal direct product representation $H = M \times C_H(M)$. We claim that the centralizer of $M$ in $H$ consists of the nonzero scalar matrices. It is straightforward that any matrix in $H$ which commutes with both
$$  X =  \left ( \begin{matrix} 1&1\\0&1 \end{matrix} \right ) \text { and }  Y = \left ( \begin{matrix} 0&1\\1&0 \end{matrix} \right )$$
must be scalar.

Now fix an element $\theta \in K$ which is not a square in $K$ and let
$$C = \left ( \begin{matrix} \theta&0\\0&1 \end{matrix} \right )$$
Then $C$ has a unique representation of the form $AB$ where $A \in M$ and $B \in C_H(M)$. 

Since $B \in C_H(M)$ the matrix for $B$ is a scalar matrix $\left ( \begin{matrix} \beta&0\\0&\beta \end{matrix} \right )$
for some nonzero $\beta \in K$. But then
$$\theta = \det \left ( \begin{matrix} \theta&0\\0&1 \end{matrix} \right ) = \det(A)\det(B) = \beta^2.$$
Then $\theta$ is a square in $K$ contrary to assumption. It follows that $K$ cannot be rigid.
 \end{proof}

  We now give the proof of Theorem 4.2.

\begin{proof} (Theorem 4.2) Let $G$ be a monolithic group with monolith isomorphic to $PSL(2,K)$ where $K$ is a field of characteristic $2$ and $|K| \ge 4$ and suppose that $G$ is CT. 

From Theorem 4.3 we have that $PSL(2,K)$ is not rigid. From [Wan] we have that an automorphism of $SL(2,K)$ is of the form $A \mapsto PA^\sigma P^{-1}$ or $A \mapsto P(A^\iota)^t P^{-1}$ where $\sigma$ is an automorphism of $K$, $\iota$ is an antiautomorphism of $K$ and $A^t$ is the transpose of $A$. Since $K$ is commutative being a field any anti-automorphism is already an automorphism. Further the transpose operator is an anti-automorphism but not an automorphism of $SL(2,K)$. It follows than that the maps of the form $A \mapsto P(A^\iota)^t P^{-1}$ do not occur here.

Since $K$ is not rigid there is a nonidentity autmorphism of $K$. Let $\sigma$ be such a nonidentity automorphism and let $\rho$ be the nontrivial automorphism of $PSL(2,K)$ given by
$$ \rho  \left ( \begin{matrix} a&b\\c&d \end{matrix} \right ) \mapsto  \left ( \begin{matrix} \sigma(a)&\sigma(b)\\\sigma(c)&\sigma(d) \end{matrix} \right ).$$
Now let $A,B$ be the elements of $SL(2,K)$ given by
$$ A = \left ( \begin{matrix} 1&1\\0&1 \end{matrix} \right ) \text{ and } B = \left ( \begin{matrix} 1&0\\1&1 \end{matrix} \right ).$$
Let $\alpha$ denote conjugation in $SL(2,K)$ by $A$ and $\beta$ denote conjugation in $SL(2,K)$ by $B$. 

By direct computation we have that in Aut$(SL(2,K)$ the automorphisms $\alpha,\beta$ both commute with $\tau$, that is $\alpha\tau = \tau\alpha$ and $\tau\beta = \beta \tau$.  However again by direct computation we have $\alpha\beta \ne \beta\alpha$. 

By Lemma 4.1 $G$ embeds in Aut($M) = $Aut$(SL(2,K)$ and it follows that the image of $G$ in the automorphism group must also be CT. This combined with the computations above imply that no transformation of the form
$$ \rho \left ( \begin{matrix} a&b\\c&d \end{matrix} \right ) \mapsto  \left ( \begin{matrix} \sigma(a)&\sigma(b)\\\sigma(c)&\sigma(d) \end{matrix} \right )$$
for a nontrivial automorphism $\sigma$ of $K$ can occur in the image of $G$ in Aut$(M)$. Thus the image of $G$ in Aut$(M)$ must consist solely of inner automorphisms of $M$. Hence for every $g \in G$ there exists an $h \in M$ such that $gxg^{-1} = hxh^{-1}$ for all $x \in M$. Therefore
$$ g^{-1}h \in C_G(M) = ker(G \rightarrow \text{Aut}(M)) = \{1\}.$$ 
Thus $g = h \in M$ and since $g,h$ were arbitrary $G = M$ completing the proof.     
\end{proof}

Finally recall that a class of groups $\cal X$ is {\bf axiomatic} is this class is defined in terms of a set of first order sentences (see [FGMRS]) or axioms. We have seen that the CT property is given by the group axioms together with
$$\text(CT: ) \forall x,y,z (((y \ne 1) \wedge ([x,y] =1) \wedge ([y,z] = 1) \rightarrow ([x,z] = 1))$$
while the class of CSA groups is captured by
 
$$\text(CT: ) \forall x,y,z (((y \ne 1) \wedge ([x,y] =1) \wedge ([y,z] = 1) \rightarrow ([x,z] = 1))$$
$$\text(MAL: ) \forall x,y,z ((x \ne 1) \wedge (y \ne 1) \wedge  ([x,y] =1) \wedge ([x,z^{-1}yz]=1) \rightarrow ([y,z] = 1)).$$

Finally being CT but not CSA is captured (in the presence of the group axioms) by
$$\text(CT: ) \forall x,y,z (((y \ne 1) \wedge ([x,y] =1) \wedge ([y,z] = 1) \rightarrow ([x,z] = 1))$$
$$\text(NOTMAL: ) \exists x,y,z ((x \ne 1) \wedge (y \ne 1) \wedge  ([x,y] =1) \wedge ([x,z^{-1}yz]=1) \wedge ([y,z] \ne 1)).$$

If $\cal{G}$ represent the class of CT groups, $\cal{H}$ the class of CSA groups and $\cal{M} = \cal{G} \cap (\cal{H})^c$ the class of CT but not CSA groups, then all three classes are axiomatic.

\begin{theorem} The class $\cal{G}$ of CT groups, the class $\cal{H}$ of CSA groups and the class  $\cal{M} = \cal{G} \cap (\cal{H})^c$ of CT non CSA groups are all axiomatic.
\end{theorem}

\section{References}

\noindent \lbrack BB] B. Baumslag, Residually free groups, \textbf{\ Proc. London Math. Soc.} (3), 17, 1967, 635 -- 645. 

\noindent \lbrack GB] G. Baumslag,  On generalised free products, \textbf {Math. Z.},78,1962,423--438.

\noindent \lbrack BeS] J. L. Bell and A. B. Slomson, Models and Ultraproducts: An Introduction,
Second Revised Printing, North-Holland, Amsterdam, 1971.

\noindent \lbrack Bo] O. Bogopolski, A surface analogue of a theorem of Magnus, \textbf{Cont. Math}, vol. 352, 2005, 55-89.

\noindent \lbrack BoS] O. Bogopolski and K. Sviridov, A Magnus theorem for some one-relator groups, in \textbf{The Zieschang Gedenkschrift}, 2008, 63-73.

\noindent \lbrack CK] C. C. Chang and H. J. Keisler, Model Theory, Second Edition,
North-Holland,Amsterdam, 1977.

\noindent \lbrack CFR] L. Ciobanu, B. Fine and G. Rosenberger, Classes of Groups Generalizing a Theorem of Benjamin Baumslag, - to appear. 

\noindent \lbrack F] B. Fine, Power Conjugacy and SQ-universality in Fuchsian and Kleinian Groups, in: \textbf{Modular Functions in Analysis and Number Theory}, University of Pittsburgh Press, 1983, 41-55.

\noindent \lbrack FMgrRR] B. Fine, A. Myasnikov, V. gr. Rebel and G. Rosenberger, A Classification of Conjugately Separated Abelian, Commutative Transitive and Restricted Gromov One-Relator Groups, \textbf{Result. Math.}, 50, 2007, 183-193.

\noindent \lbrack FGMRS] B. Fine, A. Gaglione, A. Myasnikov,
G. Rosenberger and D. Spellman, The Elementary Theory of Groups, DeGruyter, Berlin 2015 - to appear.

\noindent \lbrack FGRS]  B. Fine, A. Gaglione,
G. Rosenberger and D. Spellman, Something for Nothing: Some Consequences of the Solution to the Tarski Problems, to appear  \textbf{Groups  St. Andrews 2013}.

\noindent \lbrack FR] B. Fine, G. Rosenberger, Reflections on Commutative Transitivity, in \textbf{Aspects of Infinite 
Groups}, World Scientific Press, 2009, 112-130.

\noindent \lbrack GS]  A. Gaglione and D. Spellman, Even More Model Theory of Free Groups, in \textbf{Infinite Groups and Group Rings }edited by J.Corson, M.Dixon, M.Evans, F.Rohl, World Scientific Press, 1993, 37-40.

\noindent \lbrack GLS]  A. Gaglione, S. Lipschutz and D. Spellman, Almost Locally Free Groups and a Theorem of Magnus, \textbf{ J. groups,compleity and Cryptology},  1,  2009, 181-198.

\noindent \lbrack GKM] D. Gildenhuys, O. Kharlampovich and A. Myasnikov,  CSA Groups and Separated
Free Constructions, \textbf{ Bull. Austral. Math. Soc.}, 52,  1995,  63-84.

\noindent \lbrack Hal] M. Hall, \textbf{Group Theory}, 
Macmillan 1965.

\noindent \lbrack Ha] N. Harrison,  Real Length Functions in Groups, \textbf{Trans. Amer. Math. Soc.}, 174,  1972, 77--106.

\noindent \lbrack Ho] J. Howie,  Some Results on One-Relator Surface Groups, \textbf{ Boletin de la
Sociedad Matematica Mexicana}, 10, 2004, 255-262. 
 
\noindent \lbrack Hod] W. Hodges,  \textbf{Building Models by Games} Cambridge University Press, 1985.

\noindent \lbrack Ja] N. Jacobson, \textbf{Basic Algebra}, Dover Mathematics. 

\noindent \lbrack KhM 1] O. Kharlamapovich and A. Myasnikov, {\it  
Irreducible affine varieties over a free group: I. Irreducibility
of quadratic equations and Nullstellensatz}, \textbf{  J. of Algebra},
 200,  1998,  472-516.  

\noindent \lbrack KhM 2]  O. Kharlamapovich and A. Myasnikov, 
{\it Affine varieties over a free group: II. Systems in
triangular quasi-quadratic form and a description of residually
free groups}, \textbf{  J. of Algebra},  200,  1998,  517-569.

\noindent \lbrack KhM 3]  O. Kharlamapovich and A. Myasnikov, 
{\it The Implicit Function Theorem over Free groups }, \textbf{ J. Alg. },290,2005,
 1-203. 

\noindent \lbrack KhM 4]  O. Kharlamapovich and A. Myasnikov 
{\it Effective JSJ Decompositions } \textbf{ Cont. Math.}, 378, 2005, 87-211.

\noindent \lbrack KhM 5]  O. Kharlamapovich and A. Myasnikov, 
{\it Elementary Theory of Free Nonabelian Groups } \textbf{ J. Alg.}, 302, 2006, 451-552.

\noindent \lbrack La] S. Lang, \textbf{ Algebra}, Addison-Wessley, 1984.

\noindent \lbrack LR] F. Levin and G. Rosenberger, On Power Commutative and Commutation Transitive
Groups, \textbf{Proc. Groups St Andrews 1985}, Cambridge University Press, 1986, 249-253.
  
\noindent \lbrack LS] R. C. Lyndon and P. E. Schupp, \textbf{ Combinatorial Group Theory}, 
Springer-Verlag 1977.

\noindent \lbrack MKS] W. Magnus, A. Karrass and D. Solitar, \textbf{ Combinatorial Group Theory}, 
Wiley-Interscience 1966.

\noindent \lbrack MR] A. Myasnikov and V. Remeslennikov, Length functions on free exponential groups, \textbf{Proc. of Intern. 
Conference in Analysis and Geometry}, Omsk, 1995, 59-61.

\noindent \lbrack N] H. Neumann, \textbf{Varieties of Groups}, Springer-Verlag, 1968.
 
\noindent \lbrack O] A. Olshanskii, On Relatively Hyperbolic and G-subgroups of Hyperbolic Groups \textbf{Int. J. Alg. and Comput.} 3, 1993, 365-409.

\noindent \lbrack Re] V. N. Remeslennikov,   $\exists$-free groups \textbf{ Siberian Mat. J.},  30,
 1989,  998--1001. 

\noindent \lbrack Sc] W. R. Scott, \textbf{Group Theory}, Dover Reprints.

\noindent \lbrack Se 1] Z. Sela,  Diophantine Geometry over Groups I: Makanin-Razborov Diagrams, 
 \textbf{Publ. Math. de IHES}  93,  2001,  31-105. 

\noindent \lbrack Se 2] Z. Sela,  Diophantine Geometry over Groups II: Completions, Closures and
Fromal Solutions,  \textbf{ Israel Jour. of Math.},  104,  2003,  173-254. 

\noindent \lbrack Se 3] Z. Sela,  Diophantine Geometry over Groups III: Rigid and Solid Solutions, 
 \textbf{Israel Jour. of Math.},  147,  2005,  1-73.

\noindent \lbrack Se 4] Z. Sela,  Diophantine Geometry over Groups IV: An Itertaive Procedure for
Validation of a Sentence, \textbf{Israel Jour. of Math.}, 143,  2004,  1-71. 

\noindent \lbrack Se 5] Z. Sela,  Diophantine Geometry over Groups V: Quantifier Elimination,
 \textbf{Israel Jour. of Math.},  150,  2005,  1-9.

\noindent \lbrack Su]	M. Suzuki, The Nonexistence of a Certain Type of Simple Groups of Odd
Order, \textbf{Proc. Amer. Math. Soc.}, 8, 1957, 686-695. 

\noindent \lbrack W]	L. Weisner, Groups in which the normalizer of every element but the
identity is abelian, \textbf{Bull. Amer. Math. Soc.}, 31, 1925, 413-416. 

\noindent \lbrack Wu]	Y. F. Wu, Groups in which Commutativity is a Transitive Relation, \textbf{J. of Algebra}, 207, 1998,165-181.

\noindent \lbrack Wan] Z. K. Wan,  A proof of the automorphisms of linear groups over a field of characteristic 2,
 \textbf{Sci. Sin.},  11,  1962,  1183-1194.

\end{document}